\documentclass{amsart}

\usepackage{amsmath}
\usepackage{amsfonts}
\usepackage{amssymb}
\usepackage{latexsym}
\usepackage{amsthm}
\usepackage{setspace}
\usepackage{hyperref}
\usepackage{geometry}
\usepackage[usenames,dvipsnames]{color}
\everymath{\displaystyle}
\usepackage{url}
\usepackage{comment}

\newcommand{\Z}{\mathbb Z}
\newcommand{\Q}{\mathbb Q}

\newcommand{\F}{\mathbb F}

\newtheorem{tm}{Theorem}

\theoremstyle{remark}
\newtheorem{rem}[tm]{\bf Remark}

\DeclareMathOperator{\GL}{GL}

\newcommand{\fp}{\mathfrak{p}}
\newcommand{\cE}{\mathcal{E}}

\begin{document}

\title{Criteria for $p$-ordinarity of families of elliptic curves over infinitely many number fields}

\author{Nuno Freitas}
\address{Mathematisches Institut, Universitat Bayreuth, 95440 Bayreuth, Germany}
\email{nunobfreitas@gmail.com}

\thanks{During part of this project the first author was supported by a grant from Fundaci\'o Ferran Sunyer i Balaguer.\\The second author was supported by the SPP 1489 priority program of the Deutsche Forschungsgemeinschaft.}

\author{Panagiotis Tsaknias}
\address{
                Department of Mathematics, University of Luxembourg,
                Campus Kirchberg, 6 rue Richard Coudenhove-Kalergi,
                L-1359 Luxembourg}
\email{panagiotis.tsaknias@uni.lu}
\email{p.tsaknias@gmail.com}


\date{}

\begin{abstract} Let $K_i$ be a number field for all $i \in \Z_{> 0}$ and let $\cE$ be a family of
elliptic curves containing infinitely many members defined over $K_i$ for all $i$. Fix a rational prime $p$. We give sufficient conditions
for the existence of an integer $i_0$ such that, for all $i > i_0$ and all elliptic curve $E \in \cE$ having
good reduction at all $\fp \mid p$ in $K_i$, we have that $E$ has good ordinary reduction at all primes $\fp \mid p$.

We illustrate our criteria by applying it to certain Frey curves in \cite{F} attached to Fermat-type equations of signature $(r,r,p)$.
\end{abstract}

\maketitle

\section{Introduction}

Fix $p$ a rational prime. Let $K$ be a number field and for a prime $\fp \mid p$ write $f_\fp$ for its residual degree. Given an elliptic curve $E/K$ with good reduction at $\fp$ we know that its trace of Frobenius at $\fp$ is given by the quatity
\begin{equation}
a_\fp(E) := (p^{f_\fp} + 1) - \# \tilde{E}(\F_{p^{f_\fp}}), 
\label{traces}
\end{equation}
where $\tilde{E}$ is the reduction of $E$ modulo $\fp$. 

\bigskip

Let $E / K$ be an elliptic curve given by a Weierstrass model with good reduction at all $\fp \mid p$ in $K$. It is simple to decide whether $E$ is $p$-ordinary (i.e. has good ordinary reduction at all $\fp \mid p$). Indeed, for each $\fp \mid p$, compute $a_\fp(E)$ and check if $p \nmid a_\fp(E)$. If the previous holds for all $\fp \mid p$ then $E$ is $p$-ordinary.

Now let $(E_{\alpha} / K)_{\alpha \in \Z^n}$ be a family of elliptic curves given by their Weierstrass models. Suppose that $E_\alpha$ has good reduction at all $\fp \mid p$ for all $\alpha$. 
 Suppose further that $a_\fp(E_\alpha) = a_\fp(E_{\beta})$ for $\alpha,\beta$ whose entries are congruent modulo $p$. We are interested in deciding whether $E_\alpha$ is $p$-ordinary for all $\alpha$. This is also simple since, using formula \eqref{traces}, we only need to compute $a_\fp(E_{\alpha})$ for all the $\alpha$ that are different modulo $p$, and all $\fp \mid p$ in $K$. Then, if $p$ does not divide any of the previous values it follows that $E_\alpha$ is $p$-ordinary for all $\alpha$.

A natural generalization is to consider the same question without the assumption that the $E_\alpha$ are all defined over the same field $K$. In this note we approach this question. Indeed, we will describe sufficient conditions (see Theorem~\ref{Maintm}) to establish $p$-ordinarity of a family of elliptic curves defined over varying fields.

\bigskip

A natural source of infinite families of elliptic curves is the application of the modular method to equations of Fermat-type $Ax^p + By^r = Cz^q$. Indeed, for certain particular cases of the previous equation, it is possible to attach to a solution $(a,b,c) \in \Z^3$ a Frey elliptic curve $E_{(a,b,c)}$ given by a Weierstrass model depending on $a,b,c$. This generates an infinite family of elliptic curves. Moreover, in \cite{F} this method is applied to infinitely many equations generating a family of elliptic curves defined over varying fields. In section~3 below, we will use this family from \cite{F} to illustrate our main result.

\section{Main Theorem}

For the rest of this section we fix a rational prime $p$. For every $i\in\Z_{>0}$ let $K_i:=\Q(z_i)$ be a number field. Let $A$ be an indexing set. 
Consider a family of elliptic curves 
\[
 \mathcal{E} :=\{ E_{\alpha, i} \, : \, \alpha \in A, \, i \in \Z_{>0} \}
\]
where $E_{\alpha, i}$ is an elliptic curve defined over $K_i$  for all $\alpha \in A$  and all $i\in \Z_{>0}$. For a prime $\fp \mid p$ in $K_i$ let $E_{\alpha,i,\fp} / K_i$ be a $\fp$-minimal Weierstrass model for $E_{\alpha, i}$. Write $c_4(E_{\alpha,i,\fp})$ and $\Delta(E_{\alpha,i,\fp})$
for the usual invariants attached to the model $E_{\alpha,i,\fp}$.

\bigskip

For every $i>0$ let $A_{i}\subseteq A$ be the set of $\alpha$  for which $E_{\alpha,i}$ has good reduction at all $\fp \mid p$ in $K_i$.
Suppose further that for all $i>0$, all $\alpha\in A_{i}$ and all $\fp$ of $K_i$ above $p$ there exist polynomials $C_{\alpha,i,\fp}, D_{\alpha,i,\fp}\in L_{\alpha,i,\fp}[X]$, where $L_{\alpha,i,\fp}$ is a subfield of $K_i$ such that:
\begin{itemize}
\item $C_{\alpha,i,\fp}(z_i) = c_4(E_{\alpha,i,\fp}) \quad \text{ and } \quad  D_{\alpha,i,\fp}(z_i) = \Delta(E_{\alpha,i,\fp})$.
\item For each of the fields $ L_{\alpha,i,\fp}$ there is a unique prime $\mathfrak{q}_{\alpha, i,\fp}$ below $\fp$. Let $f_{\alpha,i,\fp}$ be its residual degree. Then the set $\{f_{\alpha,i,\fp} \}$ is bounded. This implies that there exists a finite extension $\F$ of $\F_p$ such that $\mathcal{O}_{L_{\alpha, i, \fp}}/\mathfrak{q}_{\alpha, i, \fp}$ is a subfield of $\F$ for all fields $L_{\alpha, i, \fp}$.
\item  $C_{\alpha,i,\fp}, D_{\alpha,i,\fp}\in (\mathcal{O}_{L_{\alpha, i, \fp}})_{\mathfrak{q}_{\alpha, i, \fp}}[X]$.
\item The sets $\{\deg \overline{C}_{\alpha,i,\fp}\}\text{ and }\{\deg \overline{D}_{\alpha,i,\fp}\}$ are bounded. Here $\overline{C}_{\alpha,i,\fp}$ and $\overline{D}_{\alpha,i,\fp}$ denote the reductions of $C_{\alpha,i,\fp}$ and $D_{\alpha,i,\fp}$ modulo $\mathfrak{q}_{\alpha, i, \fp}$. In particular we can assume that they are all in $\F[X]$, where $\F$ is the finite extension of $\F_p$ mentioned above.
\end{itemize}

Our main theorem is then the following:
\begin{tm}\label{Maintm}
Let $p$ be a rational prime and $(K_i)_{i \in \Z_{>0}}$ and $\mathcal{E}$ be as above. For each $\fp \mid p$ in $K_i$ let $f^i_\fp$ be the corresponding residual degree. Write $f_i$ for the minimum of the $f^i_\fp$. Suppose that
\[
 \lim_i f_i =+\infty.
\]
Then, there exists a positive integer $i_0$ such that for all $i > i_0$ and all $\alpha\in A_{i}$ the elliptic curves $E_{\alpha,i}$ are ordinary at all primes $\mathfrak{p} \mid p$ in $K_i$.
\end {tm}

\begin{proof}

Fix an algebraic closure $\overline{\F}_p$ of $\F_p$ 
and denote by $B_p$ the set of all the supersingular $j$-invariants modulo $p$ which is well known to be finite (see for example \cite[Chapter V, Theorem 4.1]{sil1}).
Moreover, from \cite[Chapter V, Theorem 3.1 ]{sil1} we have $B_p \subset \F_{p^2}\subseteq\overline{\F}_p$. Let $E$ be an elliptic curve over a number field $K$, having good reduction at (all primes above) $p$. 
For a prime $\fp \mid p$ in $K$ we write $\overline{j(E)}$ for $j(E) \pmod{\fp}$ seen as an element of $\overline{\F}_p$. Then, $\overline{j(E)} \neq b$ for all $b \in B_p$
implies that $E$ is ordinary at $\fp$. 

For the rest of the proof we fix a prime $\mathfrak{p}$ over $p$ for each $K_i$. Thus, we have that for $\alpha \in A_{i}$ the curve $E_{\alpha,i}$ is ordinary at $\fp$ if
\[
\overline{j(E_{\alpha,i})}=\frac{\overline{c_4(E_{\alpha,i,\fp})}^3}{\overline{\Delta(E_{\alpha,i,\fp})}} = \frac{\overline{C}_{\alpha,i,\fp}(\overline{z_i})^3}{\overline{D}_{\alpha,i,\fp}(\overline{z_i})}\neq b 
\]
for all $b\in B_p$, where $\overline{z}_i:=z_i \pmod{\fp}$ seen as an element of $\overline{\F}_p$.

\bigskip


Set
\begin{equation}
d:=\max_{\substack{b\in B_p\\ \alpha\in \cup_{i>0} A_{i}\\ \fp|p}}\{\deg(\overline{C}_{\alpha,i,\fp}(X)^3-b\overline{D}_{\alpha,i,\fp}(X))\}, 
\label{eqn:deg}
\end{equation}

Without loss of generality we may assume that $\F_{p^2}$ is contaned in $\F$. By assumption $d$ is finite, so there is a constant $i_0$ such that $f_i > cd$ 
for all $i > i_0$, where $c=[\F:\F_p]$. Suppose now that $E_{\alpha,j,\fp}$ over $K_j$ satisfies 
\[
\overline{C}_{\alpha,j,\fp}(\overline{z}_j)^3-b\overline{D}_{\alpha,j,\fp}(\overline{z}_j) = 0
\]
Then, the residual degree $f_{\fp}^j$ of $K_j$ at $\fp$ is at most $cd$, hence $j \leq i_0$. Thus, for all $i > i_0$, all $b\in B_p$ we have that
\[
\overline{C}_{\alpha,i,\fp}(\overline{z}_i)^3-b\overline{D}_{\alpha,i,\fp}(\overline{z}_i) \neq 0 \quad \Leftrightarrow \quad \frac{\overline{C}_{\alpha,i,\fp}(\overline{z}_i)^3}{\overline{D}_{\alpha,i,\fp}(\overline{z}_i)}\neq b.
\]
for any choice of $\fp$ in $K_i$ above $p$ and therefore we conclude that for all $i > i_0$, the curve $E_{\alpha,i}$ is ordinary at $p$ for all $\alpha \in A_{i}$.

\end{proof}

\begin{rem}\label{RemExplicit}
We can obtain an even smaller $d$ and therefore $i_0$ if we let $d$ be the maximum among the degrees of the irreducible factors of the polynomials $\overline{C}_\alpha(X)^3-b\overline{D}_\alpha(X)$ over $\F$.
If one has an explicit enough description of the residual degrees for the fields $K_i$ one can turn this in to an algorithm for explicitly computing $i_0$. This will be illustrated in the example below (see Theorem \ref{NunoFrey}).
\end{rem}

\section{Application}

First let us remark that the sequence of fields $K_r:=\Q(\zeta_r)$, indexed by rational primes $r$, with $\zeta_r$ an $r$-th primitive root of unity, satisfy the conditions of the Main Theorem. 
One family of elliptic curves that satisfies the hypotheses of our Main Theorem consists of the generalized Frey curves constructed in \cite{F}. We also give a description here for convenience:

Let $k=(k_1,k_2,k_3)$. Then for any pair $(a,b)\in \Z^2\backslash\{(0,0\}$ we can define an elliptic curve
$$E^k_{(a,b),r}: Y^2 = X(X-A^k(a,b,\zeta_r))(X+B^k(a,b,\zeta_r)).$$
Here
$$A^k(a,b,\zeta_r) := (\zeta_r^{k_3} + \zeta_r^{-k_3} - \zeta_r^{k_2} - \zeta_r^{-k_2})(a^2 + (\zeta_r^{k_1}+\zeta_r^{-k_1})ab + b^2),$$
$$B^k(a,b,\zeta_r) := (\zeta_r^{k_1} + \zeta_r^{-k_1} - \zeta_r^{k_3} - \zeta_r^{-k_3})(a^2 + (\zeta_r^{k_2}+\zeta_r^{-k_2})ab + b^2) \textrm{ and}$$
$$C^k(a,b,\zeta_r) := - A^k(a,b,\zeta_r) - B^k(a,b,\zeta_r).$$
In what follows we will refer to them as simply $A, B, C$.
\begin{tm}\label{NunoFrey}\label{Freycurve}
Assume $k=(1,2,3)$ or $(1,2,4)$. 
Then, $E^k_{(a,b),r}$ is $3$-ordinary for all primes $r>7$ and all non-zero pairs $(a,b)\in \Z^2$.
\end{tm}

\begin{proof}[{\bf Proof of Theorem \ref{Freycurve}:}]

One needs to check that the fields $K_r$ and the families of elliptic curves $E_{(a,b),r}$ satisfy indeed the hypotheses of Theorem \ref{Maintm}:
\begin{itemize}
\item The fields $K_r$  and $\Q(\zeta_r)$ are Galois and therefore the residue class degrees at 3 are all equal to the minimum. Write $f_r$ and $g_r$ for the residue class degree at 3 of $K_r$  and $\Q(\zeta_r)$, respectively. One has (see for example \cite[p. 35]{milneCFT}) that $g_r$ is the smallest positive integer $g$ such that $r|3^g-1$. This clearly implies that $\lim_r g_r=+\infty$. Since $g_r$ is equal to $f_r$ or $2f_r$ one has the corresponding property for the fields $K_r$ as well.
\item $K_r=\Q(\xi_r)$ where $\xi_r=\zeta_r+\zeta_r^{-1}$. The model (described in \cite[Section 2.3]{F}) for each curve $E_{(a,b),r}$ is given by an equation for which $c_4(E_{(a,b),r}) = 2^4(AB + BC + AC)$ and $\Delta(E_{(a,b),r})=2^4(ABC)^2$ where $A$, $B$ and $C$ are the aforementioned polynomials evaluated at $a, b, \xi_r$ and therefore the same holds for $c_4$ and $\Delta$. It is also clear from the expressions that they actually lie in $\Z_{(3)}[a,b, \xi_r]$. We therefore have that for fixed (integer) parameters $a,b$ the parameters $c_4$ and $\Delta$ are in indeed given  by  polynomials in $C_{(a,b)}, D_{(a,b)} \in \Z_{(3)}[X]$ evaluated at $\xi_r$.
\item The boundedness condition on the degrees of $\overline{C}_{(a,b)}$ and $\overline{D}_{(a,b)}$ as we let $a$ and $b$ vary is also evident from the fact that $C_{(a,b)}(X), D_{(a,b)}(X)\in \Z_{(3)}[a,b][X]$; varying $a$ and $b$ matters only up to reduction mod $3$.
\item Good reduction for the curves with $a\textrm{ or }b\not\equiv 0 \pmod{3}$ at primes above $3$ is proven in \cite[Proposition 3.2]{F}. In other words, $A_{i}=\Z^2\backslash(3\Z)^2$ for all $i$.
\end{itemize}
Theorem \ref{Maintm} thus implies that there is a constant $r_0$ such that for all $r>r_0$ all the curves are ordinary at all primes above $3$. Our goal now is to make this constant explicit, i.e. show that $r_0=7$. We proceed as outlined in Remark \ref{RemExplicit}. From here onwards $\mathfrak{p}$ will denote a prime above $3$.

For $p=3$ we have that $B_3=\{0\}\subseteq\F_3$. Thus one needs to check that $c_4^3\not\equiv 0 \mod \fp$ or equivalently that $c_4\not\equiv 0 \mod \fp$. The last one is true if and only if 
\begin{equation}
AB + BC + AC \not\equiv 0 \pmod{\mathfrak{p}}.
\label{prodsumcong}
\end{equation}
Since $A+B+C=0$, using the identity
$$
 (A+B+C)^2 = A^2 + B^2 + C^2 + 2(AB + AC + BC)
$$
we get that conguence \eqref{prodsumcong} is equivalent to 
\begin{equation}
A^2+B^2+C^2\not\equiv0\pmod{\mathfrak{p}}.
\label{squarescong}  
\end{equation}
Notice that $AB + BC + AC\pmod{\mathfrak{p}}$ depends only on $(a,b)\pmod{3}$ so we will assume from now on that $(a,b)\in \F_3^2\backslash\{(0,0)\}$. Furthermore, by the symmetry of $A$, $B$, $C$, it is enough to consider only the cases where $(a,b) \in \{(1,0),(1,1),(1,2)\}$.
Assume for now (which is going to be true for the cases we will consider) that we can find $u,v,w$ such that 
\begin{equation}
A=v-w, \quad B=w-u, \quad C=u-v.
\label{dif}
\end{equation}
Then congruence (\ref{squarescong}) is equivalent to 
\begin{equation}
(v-w)^2+(w-u)^2+(u-v)^2\not\equiv0\pmod{\mathfrak{p}}.
\label{squaresdifcong} 
\end{equation}
Since $\mathfrak{p} \mid 3$ we have that $u^3 + v^3 + w^3 \equiv (u+v+w)^3 \pmod{\mathfrak{p}}$ and therefore 
\begin{equation}
u+v+w\not\equiv0\pmod{\mathfrak{p}},
\label{sumcong} 
\end{equation}
is equivalent to $u^3+v^3+w^3\not\equiv0\pmod{\mathfrak{p}}$. Furthermore, using the identity 
$$u^3+v^3+w^3=\frac{1}{2}(u+v+w)\left[(w-v)^2 + (u-w)^2 + (v-u)^2\right]+3uvw,$$
we see that congruence \eqref{sumcong} implies congruence \eqref{squaresdifcong}. The values of $u$, $v$ and $w$ in each of the three cases for $(a,b)$ are:
\begin{itemize}
\item {\bf The case $(a,b)=(1,0)$.}
In this case we have
$$A = \xi_{k_3} - \xi_{k_2},\quad B = \xi_{k_1} - \xi_{k_3},\quad C = \xi_{k_2} - \xi_{k_1}$$
and it is trivial to see that
$$u=\xi_{k_1}, \quad v=\xi_{k_2},  \quad w=\xi_{k_3}.$$
\item {\bf The case $(a,b)=(1,1)$.}
In this case we have
$$A = (\xi_{k_3} - \xi_{k_2})(2 + \xi_{k_1}),\quad B = (\xi_{k_1} - \xi_{k_3})(2 + \xi_{k_2}),\quad C = (\xi_{k_2} - \xi_{k_1})(2 + \xi_{k_3})$$
and it is easy to see that
$$u=\xi_{k_2}\xi_{k_3} -2\xi_{k_1}, \quad v = \xi_{k_1}\xi_{k_3} -2\xi_{k_2}, \quad w=\xi_{k_1}\xi_{k_2} -2\xi_{k_3}.$$
\item {\bf The case $(a,b)=(1,2)$.}
In this case we have
$$ A = (\xi_{k_3} - \xi_{k_2})(5 + 2\xi_{k_1}),\quad B = (\xi_{k_1} - \xi_{k_3})(5 + 2\xi_{k_2}),\quad C = (\xi_{k_2} - \xi_{k_1})(5 + 2\xi_{k_3})$$
and it is easy to see that
$$u=2\xi_{k_2}\xi_{k_3} -5\xi_{k_1}, \quad v = 2\xi_{k_1}\xi_{k_3} -5\xi_{k_2}, \quad  w = 2\xi_{k_1}\xi_{k_2} -5\xi_{k_3}.$$
\end{itemize}
It is easy to see that one can write $u+v+w$ as $h(\xi_1)$ with $h(X)\in \Z[X]$ using the identities:
$$\xi_k=\xi_1^k - \sum_{j=1}^{\lfloor k/2\rfloor}\binom{k}{j}\xi_{k-2j}\quad{\rm for\ } k {\rm\ odd\ and}$$
$$\xi_k=\xi_1^k - \sum_{j=1}^{k/2-1}\binom{k}{j}\xi_{k-2j}-\binom{k}{k/2}\quad{\rm for\ } k {\rm\ even}.$$
Notice that the degree of $h$ depends on the triple $(k_1,k_2,k_3)$ and $(a,b)$ but not on $r$.

Assume now that congruence \eqref{sumcong} is not true, i.e. that
$h(\xi_i)\equiv0\pmod{\mathfrak{p}}$. Then $g(\zeta_r)\equiv0\pmod{\mathfrak{p}}$ where $g(X) \in \Z[X]$ is the polynomial $\smash{X^{\deg(h)}h(X+1/X)}$, of degree $d=2\deg(h)$, still independent of $r$. This implies that the extension $\F_3[\overline{\zeta_r}]/\F_3$ is of degree at most $d$.

The relation $r|3^f-1$ implies that, for a fixed $f$, there are only finitely many $r$, easily explicitly determined, such that the residue class degree is (at most) $f$. To finish things, we just have to examine what happens at these exceptional $r$. We will do this for $(a,b)=(1,1)$ and $(k_1,k_2,k_3)=(1,2,3)$: In this case $h$ is of degree $5$ and it factors in $\F_3[X]$ as
$$(1+X)(2+X)(2+X+X^2+X^3).$$
The only primes $r\geq7$ for which the extension $\F_3[\zeta_r]/\F_3$ is of degree at most $6$ are $7$, $11$ and $13$. One then verifies computationally for these primes that the Frey curve is indeed $3$-ordinary, except for $7$. We again look at (the prime divisors of) the norm of $u+v+w$: 
\begin{itemize}
\item $r=7$: The norm is $0$.
\item $r=11$: The norm is $11^2$.
\item $r=13$: The norm is $13^2$.
\end{itemize}

The other cases are treated the same way and it turns out that $r>7$ is the sufficient condition for both triples.
\end{proof}

\section*{Acknowledgments}

The authors would like to thank Sara Arias de Reyna and Samir Siksek for useful suggestions. The first author would like to thank Gabor Wiese and the University of Luxembourg for the pleasant stay there, during which part of this work was done. Both authors would like to thank Fred Diamond and King's College London for the same reasons.


\begin{thebibliography}{0}

\bibitem{F} N. Freitas, Recipes for Fermat-type equations of signature $(r, r, p)$ (preprint). \\ \url{arxiv.org/abs/1203.3371}, 31 October 2013.

\bibitem{milneCFT} J. S. Milne, Class Field Theory (v4.01), 2011. \\ 
\url{www.jmilne.org/math/}

\bibitem{sil1} J. H. Silverman, The Arithmetic of Elliptic Curves, {\it Graduate Texts in Mathematics}, Springer, Dordrecht, second edition, 2009.

\end{thebibliography}
\end{document}